\documentclass{amsart}
\usepackage{amsfonts}

\newtheorem{theorem}{Theorem}
\theoremstyle{plain}

\newtheorem{corollary}{Corollary}

\newtheorem{example}{Example}

\newtheorem{lemma}{Lemma}

\newtheorem{remark}{Remark}

\numberwithin{equation}{section}
\input{tcilatex}

\begin{document}
\title[H\"{o}lder integral inequality]{A New Improvement of H\"{o}lder
inequality via Isotonic Linear Functionals}
\author{\.{I}mdat \.{I}\c{s}can}
\address{Department of Mathematics, Faculty of Arts and Sciences,\\
Giresun University, 28200, Giresun, Turkey.}
\email{imdati@yahoo.com, imdat.iscan@giresun.edu.tr}
\subjclass[2000]{Primary 26D15; Secondary 26A51}
\keywords{H\"{o}lder Inequality, Young Inequality, Integral Inequalities,
Hermite-Hadamard Type Inequality}

\begin{abstract}
In this paper, new improvement of celebrated H\"{o}lder inequality by means
of isotonic linear functionals is established. An important feature of the
new inequality obtained in here is that many existing inequalities related
to the H\"{o}lder inequality can be improved via new improvement of H\"{o}%
lder inequality. We also show this in an application.
\end{abstract}

\maketitle

\section{Introduction}

\bigskip The famous Young's inequality, as a classical result, state that:
if $a,b>0$ and $t\in \lbrack 0,1]$, then%
\begin{equation}
a^{t}b^{1-t}\leq ta+(1-t)b  \label{0-1}
\end{equation}%
with equality if and only if $a=b.$ Let $p,q>1$ such that $1/p+1/q=1$. The
inequality (\ref{0-1}) can be written as%
\begin{equation}
ab\leq \frac{a^{p}}{p}+\frac{b^{q}}{q}  \label{0-2}
\end{equation}%
for any $x,y\geq 0$. In this form, the inequality (\ref{0-2}) was used to
prove the celebrated H\"{o}lder inequality. One of the most important
inequalities of analysis is H\"{o}lder's inequality. It contributes wide
area of pure and applied mathematics and plays a key role in resolving many
problems in social science and cultural science as well as in natural
science.

\begin{theorem}[H\"{o}lder Inequality for Integrals \protect\cite{MPF13}]
Let $p>1$ and $1/p+1/q=1$. If $f\ $and $g$ are real functions defined on $%
\left[ a,b\right] $ and if $\left\vert f\right\vert ^{p},\left\vert
g\right\vert ^{q}$ are integrable functions on $\left[ a,b\right] $ then%
\begin{equation}
\int_{a}^{b}\left\vert f(x)g(x)\right\vert dx\leq \left(
\int_{a}^{b}\left\vert f(x)\right\vert ^{p}dx\right) ^{1/p}\left(
\int_{a}^{b}\left\vert g(x)\right\vert ^{q}dx\right) ^{1/q},  \label{1-1}
\end{equation}%
with equality holding if and only if $A\left\vert f(x)\right\vert
^{p}=B\left\vert g(x)\right\vert ^{q}$ almost everywhere, where $A$ and $B$
are constants.
\end{theorem}

\begin{theorem}[H\"{o}lder Inequality for Sums \protect\cite{MPF13}]
Let $a=\left( a_{1},...,a_{n}\right) $ and $b=\left( b_{1},...,b_{n}\right) $
be two positive n-tuples and $p,q>1$ such that $1/p+1/q=1.$ Then we have%
\begin{equation}
\sum_{k=1}^{n}a_{k}b_{k}\leq \left( \sum_{k=1}^{n}a_{k}^{p}\right)
^{1/p}\left( \sum_{k=1}^{n}b_{k}^{q}\right) ^{1/q}.  \label{1-2}
\end{equation}%
Equality hold in (\ref{1-2}) if and only if $a^{p}$ and $b^{q}$ are
proportional.
\end{theorem}

In \cite{I19}, \.{I}\c{s}can gave new improvements for integral ans sum
forms of the H\"{o}lder inequality as follow:

\begin{theorem}
Let $p>1$ and $\frac{1}{p}+\frac{1}{q}=1$. If $f$ and $g$ are real functions
defined on interval $\left[ a,b\right] $ and if $\left\vert f\right\vert
^{p} $, $\left\vert g\right\vert ^{q}$ are integrable functions on $\left[
a,b\right] $ then 
\begin{eqnarray}
\int_{a}^{b}\left\vert f(x)g(x)\right\vert dx &\leq &\frac{1}{b-a}\left\{
\left( \int_{a}^{b}(b-x)\left\vert f(x)\right\vert ^{p}dx\right) ^{\frac{1}{p%
}}\left( \int_{a}^{b}(b-x)\left\vert g(x)\right\vert ^{q}dx\right) ^{\frac{1%
}{q}}\right.  \notag \\
&&\left. +\left( \int_{a}^{b}(x-a)\left\vert f(x)\right\vert ^{p}dx\right) ^{%
\frac{1}{p}}\left( \int_{a}^{b}(x-a)\left\vert g(x)\right\vert ^{q}dx\right)
^{\frac{1}{q}}\right\}  \label{1-5}
\end{eqnarray}%
\textit{\ }
\end{theorem}

\begin{theorem}
\label{2.3}Let $a=\left( a_{1},...,a_{n}\right) $ and $b=\left(
b_{1},...,b_{n}\right) $ be two positive n-tuples and $p,q>1$ such that $%
1/p+1/q=1.$ Then%
\begin{eqnarray}
\sum_{k=1}^{n}a_{k}b_{k} &\leq &\frac{1}{n}\left\{ \left(
\sum_{k=1}^{n}ka_{k}^{p}\right) ^{1/p}\left( \sum_{k=1}^{n}kb_{k}^{q}\right)
^{1/q}\right.   \label{2-3} \\
&&\left. +\left( \sum_{k=1}^{n}\left( n-k\right) a_{k}^{p}\right)
^{1/p}\left( \sum_{k=1}^{n}\left( n-k\right) b_{k}^{q}\right) ^{1/q}\right\}
.  \notag
\end{eqnarray}
\end{theorem}

\section{H\"{o}lder's inequality for positive functionals}

Let $E$ be a nonempty set and $L$ be a linear class of real valued functions
on $E$ having the following properties

$L1:$ If $f,g\in L$ then $\left( \alpha f+\beta g\right) \in L$ for all $%
\alpha ,\beta \in 
%TCIMACRO{\U{211d} }%
%BeginExpansion
\mathbb{R}
%EndExpansion
$;

$L2:$ $1\in L$, that is if $f(t)=1,t\in E,$ then $f\in L;$

$L3:$ If $f\in L,E_{1}\in L$ then $f\chi _{E_{1}}\in L,$

where $\chi _{E_{1}}$ is the indicator function of $E_{1}$. It follows from $%
L2$ and $L3$ that $\chi _{E_{1}}\in L$ for every $E_{1}\in L.$

We also consider positive isotonic linear functionals $A:L\rightarrow 
%TCIMACRO{\U{211d} }%
%BeginExpansion
\mathbb{R}
%EndExpansion
$ is a functional satisfying the following properties:

$A1:$ $A\left( \alpha f+\beta g\right) =\alpha A\left( f\right) +\beta $ $%
A\left( g\right) $ for $f,g\in L$ and $\alpha ,\beta \in 
%TCIMACRO{\U{211d} }%
%BeginExpansion
\mathbb{R}
%EndExpansion
;$

$A2:$ If $f\in L,$ $f(t)\geq 0$ on $E$ then $A\left( f\right) \geq 0.$

Furthermore, It follows from $L3$ that for every $E_{1}\in L$ such that $%
A(\chi _{E_{1}})>0,$ the functional $A_{E_{1}}$ is defined for all $f\in L$
by $A_{E_{1}}(f)=A\left( f\chi _{E_{1}}\right) /A\left( \chi _{E_{1}}\right) 
$ is a fixed positive isotonic linear functional with $A_{E_{1}}(\mathbf{1}%
)=1.$ We observe that%
\begin{equation*}
A\left( \chi _{E_{1}}\right) +A\left( \chi _{E\backslash E_{1}}\right) =1,
\end{equation*}%
\begin{equation*}
A(f)=A\left( f.\chi _{E_{1}}\right) +A\left( f.\chi _{E\backslash
E_{1}}\right) .
\end{equation*}%
Isotonic, that is, order-preserving, linear functionals are natural objects
in analysis which enjoy a number of convenient properties. Functional
versions of well-known inequalities and related results could be found in 
\cite{APV04,C14,C16,D03,D17,DKA16,MPF13,P91}.

\begin{example}
\textbf{i.) } If $E=\left[ a,b\right] \subseteq 
%TCIMACRO{\U{211d} }%
%BeginExpansion
\mathbb{R}
%EndExpansion
$ and $L=L\left[ a,b\right] ,$ then 
\begin{equation*}
A(f)=\int_{a}^{b}f(t)dt
\end{equation*}%
is an isotonic linear functional.

\textbf{ii.) }If $E=\left[ a,b\right] \times \left[ c,d\right] \subseteq 
%TCIMACRO{\U{211d} }%
%BeginExpansion
\mathbb{R}
%EndExpansion
^{2}$ and $L=L\left( \left[ a,b\right] \times \left[ c,d\right] \right) ,$
then 
\begin{equation*}
A(f)=\int_{a}^{b}\int_{c}^{d}f(x,y)dxdy
\end{equation*}%
is an isotonic linear functional.

\textbf{iii.) }If $\left( E,\Sigma ,\mu \right) $ is a measure space with $%
\mu $ positive measure on $E$ and $L=L(\mu )$ then 
\begin{equation*}
A(f)=\int_{E}fd\mu \text{ }
\end{equation*}%
is an isotonic linear functional.

\textbf{iv.) }If $E$ is a subset of the natural numbers $%
%TCIMACRO{\U{2115} }%
%BeginExpansion
\mathbb{N}
%EndExpansion
$ with all $p_{k}\geq 0,$ then $A(f)=\sum_{k\in E}p_{k}f_{k}$ is an isotonic
linear functional. For example; If $E=\left\{ 1,2,...,n\right\} $ and $%
f:E\rightarrow 
%TCIMACRO{\U{211d} }%
%BeginExpansion
\mathbb{R}
%EndExpansion
,f(k)=a_{k},$ then $A(f)=\sum_{k=1}^{n}a_{k}$ is an isotonic linear
functional. If $E=\left\{ 1,2,...,n\right\} \times \left\{ 1,2,...,m\right\} 
$ and $f:E\rightarrow 
%TCIMACRO{\U{211d} }%
%BeginExpansion
\mathbb{R}
%EndExpansion
,f(k,l)=a_{k,l},$ then $A(f)=\sum_{k=1}^{n}\sum_{l=1}^{m}a_{k,l}$ is an
isotonic linear functional.
\end{example}

\begin{theorem}[H\"{o}lder's inequality for isotonic functionals 
\protect\cite{PPT92}]
\label{2T1}Let $L$ satisfy conditions $L1$, $L2$, and $A$ satisfy conditions 
$A1$, $A2$ on a base set $E$. Let $p>1$ and $p^{-1}+q^{-1}=1.$ If $w,f,g\geq
0$ on $E$ and $wf^{p},wg^{q},wfg\in L$ then we have%
\begin{equation}
A\left( wfg\right) \leq A^{1/p}\left( wf^{p}\right) A^{1/q}\left(
wg^{q}\right) .  \label{2T1-1}
\end{equation}%
In the case $0<p<1$ and $A\left( wg^{q}\right) >0$ (or $p<0$ and $A\left(
wf^{p}\right) >0$), the inequality in (\ref{2T1-1}) is reversed.
\end{theorem}

\begin{remark}
i.) If we choose $E=\left[ a,b\right] \subseteq 
%TCIMACRO{\U{211d} }%
%BeginExpansion
\mathbb{R}
%EndExpansion
$, $L=L\left[ a,b\right] $, $w=1$ on $E$ and $A(f)=\int_{a}^{b}\left\vert
f(t)\right\vert dt$ in the Theorem \ref{2T1}, then the inequality (\ref%
{2T1-1}) reduce the inequality (\ref{1-1}).

ii.) If we choose $E=\left\{ 1,2,...,n\right\} ,$ $w=1$ on $E$, $%
f:E\rightarrow \left[ 0,\infty \right) ,f(k)=a_{k},$ and $%
A(f)=\sum_{k=1}^{n}a_{k}$ in the Theorem \ref{2T1}, then the inequality (\ref%
{2T1-1}) reduce the inequality (\ref{1-2}).

iii.) If we choose $E=\left[ a,b\right] \times \left[ c,d\right] ,L=L(E)$, $%
w=1$ on $E$ and $A(f)=\int_{a}^{b}\int_{c}^{d}\left\vert f(x,y)\right\vert
dxdy$ in the Theorem \ref{2T1}, then the inequality (\ref{2T1-1}) reduce the
following inequality for double integrals:%
\begin{equation*}
\int_{a}^{b}\int_{c}^{d}\left\vert f(x,y)\right\vert \left\vert
g(x,y)\right\vert dxdy\leq \left( \int_{a}^{b}\int_{c}^{d}\left\vert
f(x,y)\right\vert ^{p}dx\right) ^{1/p}\left(
\int_{a}^{b}\int_{c}^{d}\left\vert g(x,y)\right\vert ^{q}dx\right) ^{1/q}.
\end{equation*}
\end{remark}

The aim of this paper is to give a new general improvement of H\"{o}lder
inequality for isotonic linear functional. As applications, this new
inequality will be rewritten for several important particular cases of
isotonic linear functionals. Also, we give an application to show that
improvement is hold for double integrals.

\section{Main results}

\begin{theorem}
\label{3T1}Let $L$ satisfy conditions $L1$, $L2$, and $A$ satisfy conditions 
$A1$, $A2$ on a base set $E$. Let $p>1$ and $p^{-1}+q^{-1}=1.$ If $\alpha
,\beta ,w,f,g\geq 0$ on $E$ and $\alpha wfg,\beta wfg,\alpha wf^{p},\alpha
wg^{q},\beta wf^{p},\beta wg^{q},wfg\in L$ then we have

i.)%
\begin{equation}
A\left( wfg\right) \leq A^{1/p}\left( \alpha wf^{p}\right) A^{1/q}\left(
\alpha wg^{q}\right) +A^{1/p}\left( \beta wf^{q}\right) A^{1/q}\left( \beta
wg^{q}\right)  \label{3T1-1}
\end{equation}

ii.)%
\begin{equation}
A^{1/p}\left( \alpha wf^{p}\right) A^{1/q}\left( \alpha wg^{q}\right)
+A^{1/p}\left( \beta wf^{p}\right) A^{1/q}\left( \beta wg^{q}\right) \leq
A^{1/p}\left( wf^{p}\right) A^{1/q}\left( wg^{q}\right) .  \label{3T1-2}
\end{equation}
\end{theorem}

\begin{proof}
i.) By using of H\"{o}lder inequality for isotonic functionals in (\ref%
{2T1-1}) and linearity of $A$, it is easily seen that%
\begin{eqnarray*}
A\left( wfg\right) &=&A\left( \alpha wfg+\beta wfg\right) =A\left( \alpha
wfg\right) +A\left( \beta wfg\right) \\
&\leq &A^{1/p}\left( \alpha wf^{p}\right) A^{1/q}\left( \alpha wg^{q}\right)
+A^{1/p}\left( \beta wf^{p}\right) A^{1/q}\left( \beta wg^{q}\right) .
\end{eqnarray*}

ii.) Firstly, we assume that $A^{1/p}\left( wf^{p}\right) A^{1/q}\left(
wg^{q}\right) \neq 0$. then%
\begin{eqnarray*}
&&\frac{A^{1/p}\left( \alpha wf^{p}\right) A^{1/q}\left( \alpha
wg^{q}\right) +A^{1/p}\left( \beta wf^{p}\right) A^{1/q}\left( \beta
wg^{q}\right) }{A^{1/p}\left( wf^{p}\right) A^{1/q}\left( wg^{q}\right) } \\
&=&\left( \frac{A\left( \alpha wf^{p}\right) }{A\left( wf^{p}\right) }%
\right) ^{1/p}\left( \frac{A\left( \alpha wg^{q}\right) }{A\left(
wg^{q}\right) }\right) ^{1/q}+\left( \frac{A\left( \beta wf^{p}\right) }{%
A\left( wf^{p}\right) }\right) ^{1/p}\left( \frac{A\left( \beta
wg^{q}\right) }{A\left( wg^{q}\right) }\right) ^{1/q},
\end{eqnarray*}%
By the inequality (\ref{0-1}) and linearity of $A$, we have%
\begin{eqnarray*}
&&\frac{A^{1/p}\left( \alpha wf^{p}\right) A^{1/q}\left( \alpha
wg^{q}\right) +A^{1/p}\left( \beta wf^{p}\right) A^{1/q}\left( \beta
wg^{q}\right) }{A^{1/p}\left( wf^{p}\right) A^{1/q}\left( wg^{q}\right) } \\
&\leq &\frac{1}{p}\left[ \frac{A\left( \alpha wf^{p}\right) }{A\left(
wf^{p}\right) }+\frac{A\left( \beta wf^{p}\right) }{A\left( wf^{p}\right) }%
\right] +\frac{1}{q}\left[ \frac{A\left( \alpha wg^{q}\right) }{A\left(
wg^{q}\right) }+\frac{A\left( \beta wg^{q}\right) }{A\left( wg^{q}\right) }%
\right] \\
&=&1.
\end{eqnarray*}

Finally, suppose that $A^{1/p}\left( wf^{p}\right) A^{1/q}\left(
wg^{q}\right) =0$. Then $A^{1/p}\left( wf^{p}\right) =0$ or $A^{1/q}\left(
wg^{q}\right) =0$, i.e. $A\left( wf^{p}\right) =0$ or $A\left( wg^{q}\right)
=0.$ We assume that $A\left( wf^{p}\right) =0$. Then by using linearity of $%
A $ we have,%
\begin{equation*}
0=A\left( wf^{p}\right) =A\left( \alpha wf^{p}+\beta wf^{p}\right) =A\left(
\alpha wf^{p}\right) +A\left( \beta wf^{p}\right) .
\end{equation*}%
Since $A\left( \alpha wf\right) ,A\left( \beta wf\right) \geq 0$, we get $%
A\left( \alpha wf^{p}\right) =0$ and $A\left( \beta wf^{p}\right) =0.$ From
here, it follows that 
\begin{equation*}
A^{1/p}\left( \alpha wf^{p}\right) A^{1/q}\left( \alpha wg^{q}\right)
+A^{1/p}\left( \beta wf^{p}\right) A^{1/q}\left( \beta wg^{q}\right) =0\leq
0=A^{1/p}\left( wf^{p}\right) A^{1/q}\left( wg^{q}\right) .
\end{equation*}%
In case of $A\left( wg^{q}\right) =0,$ the proof is done similarly. This
completes the proof.
\end{proof}

\begin{remark}
The inequality (\ref{3T1-2}) shows that the inequality (\ref{3T1-1}) is
better than the inequality (\ref{2T1-1}).
\end{remark}

If we take $w=1$ on $E$ in the Theorem \ref{3T1}, then we can give the
following corollary:

\begin{corollary}
\label{3C1}\bigskip Let $L$ satisfy conditions $L1$, $L2$, and $A$ satisfy
conditions $A1$, $A2$ on a base set $E$. Let $p>1$ and $p^{-1}+q^{-1}=1.$ If 
$\alpha ,\beta ,f,g\geq 0$ on $E$ and $\alpha fg,\beta fg,\alpha
f^{p},\alpha g^{q},\beta f^{p},\beta g^{q},fg\in L$ then we have

i.)%
\begin{equation}
A\left( fg\right) \leq A^{1/p}\left( \alpha f^{p}\right) A^{1/q}\left(
\alpha g^{q}\right) +A^{1/p}\left( \beta f^{q}\right) A^{1/q}\left( \beta
g^{q}\right)  \label{3C1-1}
\end{equation}%
ii.)%
\begin{equation*}
A^{1/p}\left( \alpha f^{p}\right) A^{1/q}\left( \alpha g^{q}\right)
+A^{1/p}\left( \beta f^{p}\right) A^{1/q}\left( \beta g^{q}\right) \leq
A^{1/p}\left( f^{p}\right) A^{1/q}\left( g^{q}\right) .
\end{equation*}
\end{corollary}

\begin{remark}
i.) If we choose $E=\left[ a,b\right] \subseteq 
%TCIMACRO{\U{211d} }%
%BeginExpansion
\mathbb{R}
%EndExpansion
$, $L=L\left[ a,b\right] $, $\alpha (t)=\frac{b-t}{b-a},\beta (t)=\frac{t-a}{%
b-a}$ on $E$ and $A(f)=\int_{a}^{b}\left\vert f(t)\right\vert dt$ in the
Corollary \ref{3C1}, then the inequality (\ref{3C1-1}) reduce the inequality
(\ref{1-5}).

ii.) If we choose $E=\left\{ 1,2,...,n\right\} ,$ $\alpha (k)=\frac{k}{n}%
,\beta (k)=\frac{n-k}{n}$ on $E$, $f:E\rightarrow \left[ 0,\infty \right)
,f(k)=a_{k},$ and $A(f)=\sum_{k=1}^{n}a_{k}$ in the Theorem\ref{3C1}, then
the inequality (\ref{3C1-1}) reduce the inequality (\ref{2-3}).
\end{remark}

We can give more general form of the Theorem \ref{3T1} as follows:

\begin{theorem}
\label{3T2}Let $L$ satisfy conditions $L1$, $L2$, and $A$ satisfy conditions 
$A1$, $A2$ on a base set $E$. Let $p>1$ and $p^{-1}+q^{-1}=1.$ If $\alpha
_{i},w,f,g\geq 0$ on $E,$ $\alpha _{i}wfg,\alpha _{i}wf^{p},\alpha
_{i}wg^{q},wfg\in L,i=1,2,...,m,$ and $\sum_{i=1}^{m}\alpha _{i}=1,$ then we
have

i.)%
\begin{equation*}
A\left( wfg\right) \leq \sum_{i=1}^{m}A^{1/p}\left( \alpha _{i}wf^{p}\right)
A^{1/q}\left( \alpha _{i}wg^{q}\right)
\end{equation*}

ii.)%
\begin{equation*}
\sum_{i=1}^{m}A^{1/p}\left( \alpha _{i}wf^{p}\right) A^{1/q}\left( \alpha
_{i}wg^{q}\right) \leq A^{1/p}\left( wf^{p}\right) A^{1/q}\left(
wg^{q}\right) .
\end{equation*}
\end{theorem}

\begin{proof}
The proof can be easily done similarly to the proof of Theorem \ref{3T1}.
\end{proof}

If we take $w=1$ on $E$ in the Theorem \ref{3T1}, then we can give the
following corollary:

\begin{corollary}
\label{3C2}\bigskip Let $L$ satisfy conditions $L1$, $L2$, and $A$ satisfy
conditions $A1$, $A2$ on a base set $E$. Let $p>1$ and $p^{-1}+q^{-1}=1.$ If 
$\alpha _{i},f,g\geq 0$ on $E,$ $\alpha _{i}fg,\alpha _{i}f^{p},\alpha
_{i}g^{q},fg\in L,i=1,2,...,m,$ and $\sum_{i=1}^{m}\alpha _{i}=1,$ then we
have

i.)%
\begin{equation}
A\left( fg\right) \leq \sum_{i=1}^{m}A^{1/p}\left( \alpha _{i}f^{p}\right)
A^{1/q}\left( \alpha _{i}g^{q}\right)  \label{3C2-1}
\end{equation}%
ii.)%
\begin{equation*}
\sum_{i=1}^{m}A^{1/p}\left( \alpha _{i}f^{p}\right) A^{1/q}\left( \alpha
_{i}g^{q}\right) \leq A^{1/p}\left( f^{p}\right) A^{1/q}\left( g^{q}\right) .
\end{equation*}
\end{corollary}

\begin{corollary}[Improvement of H\"{o}lder inequality for double integrals]

Let $p,q>1$ and $1/p+1/q=1$. If $f\ $and $g$ are real functions defined on $%
E=\left[ a,b\right] \times \left[ c,d\right] $ and if $\left\vert
f\right\vert ^{p},\left\vert g\right\vert ^{q}\in L(E)$ then%
\begin{equation}
\int_{a}^{b}\int_{c}^{d}\left\vert f(x,y)\right\vert \left\vert
g(x,y)\right\vert dxdy\leq \sum_{i=1}^{4}\left(
\int_{a}^{b}\int_{c}^{d}\alpha _{i}(x,y)\left\vert f(x,y)\right\vert
^{p}dx\right) ^{1/p}\left( \int_{a}^{b}\int_{c}^{d}\alpha
_{i}(x,y)\left\vert g(x,y)\right\vert ^{q}dx\right) ^{1/q},  \label{3C3-1}
\end{equation}%
where $\alpha _{1}(x,y)=\frac{\left( b-x\right) \left( d-y\right) }{\left(
b-a\right) \left( d-c\right) },\alpha _{2}(x,y)=\frac{\left( b-x\right)
\left( y-c\right) }{\left( b-a\right) \left( d-c\right) },\alpha _{3}(x,y)=%
\frac{\left( x-a\right) \left( y-c\right) }{\left( b-a\right) \left(
d-c\right) },,\alpha _{4}(x,y)=\frac{\left( x-a\right) \left( d-y\right) }{%
\left( b-a\right) \left( d-c\right) }$ on $E$
\end{corollary}

\begin{proof}
If we choose $E=\left[ a,b\right] \times \left[ c,d\right] \subseteq 
%TCIMACRO{\U{211d} }%
%BeginExpansion
\mathbb{R}
%EndExpansion
^{2}$, $L=L(E)$, $\alpha _{1}(x,y)=\frac{\left( b-x\right) \left( d-y\right) 
}{\left( b-a\right) \left( d-c\right) },\alpha _{2}(x,y)=\frac{\left(
b-x\right) \left( y-c\right) }{\left( b-a\right) \left( d-c\right) },\alpha
_{3}(x,y)=\frac{\left( x-a\right) \left( y-c\right) }{\left( b-a\right)
\left( d-c\right) },\alpha _{4}(x,y)=\frac{\left( x-a\right) \left(
d-y\right) }{\left( b-a\right) \left( d-c\right) }$ on $E$ and $%
A(f)=\int_{a}^{b}\int_{c}^{d}\left\vert f(x,y)\right\vert dxdy$ in the
Corollary \ref{3C1}, then we get the inequality (\ref{3C3-1}).
\end{proof}

\begin{corollary}
Let $\left( a_{k,l}\right) $ and $\left( b_{k,l}\right) $ be two tuples of
positive numbers and $p,q>1$ such that $1/p+1/q=1.$ Then we have%
\begin{equation}
\sum_{k=1}^{n}\sum_{l=1}^{m}a_{k,l}b_{k,l}\leq \sum_{i=1}^{4}\left(
\sum_{k=1}^{n}\sum_{l=1}^{m}\alpha _{i}(k,l)a_{k,l}^{p}\right) ^{1/p}\left(
\sum_{k=1}^{n}\sum_{l=1}^{m}\alpha _{i}(k,l)b_{k,l}^{q}\right) ^{1/q},
\label{3C4-1}
\end{equation}%
where $\alpha _{1}(k,l)=\frac{kl}{nm},\alpha _{2}(k,l)=\frac{\left(
n-k\right) l}{nm},\alpha _{3}(k,l)=\frac{\left( n-k\right) \left( m-l\right) 
}{nm},\alpha _{4}(k,l)=\frac{k\left( m-l\right) }{nm}$ on $E.$
\end{corollary}

\begin{proof}
If we choose $E=\left\{ 1,2,...,n\right\} \times \left\{ 1,2,...,m\right\} ,$
$\alpha _{1}(k,l)=\frac{kl}{nm},\alpha _{2}(k,l)=\frac{\left( n-k\right) l}{%
nm},\alpha _{3}(k,l)=\frac{\left( n-k\right) \left( m-l\right) }{nm},\alpha
_{4}(k,l)=\frac{k\left( m-l\right) }{nm}$ on $E$, $f:E\rightarrow \left[
0,\infty \right) ,f(k,l)=a_{k,l},$ and $A(f)=\sum_{k=1}^{n}%
\sum_{l=1}^{m}a_{k,l}$ in the Theorem\ref{3C1}, then we get the inequality (%
\ref{3C4-1}).
\end{proof}

\section{\protect\bigskip An Application for Double Integrals}

In \cite{ZSOD12}, Sar\i kaya et al. gave the following lemma for obtain main
results.

\begin{lemma}
\label{L-3.1}Let $f:\Delta \subseteq 
%TCIMACRO{\U{211d} }%
%BeginExpansion
\mathbb{R}
%EndExpansion
^{2}\rightarrow 
%TCIMACRO{\U{211d} }%
%BeginExpansion
\mathbb{R}
%EndExpansion
$ be a partial differentiable mapping on $\Delta =\left[ a,b\right] \times %
\left[ c,d\right] $ in $%
%TCIMACRO{\U{211d} }%
%BeginExpansion
\mathbb{R}
%EndExpansion
^{2}$with $a<b$ and $c<d.$ If $\frac{\partial ^{2}f}{\partial t\partial s}%
\in L(\Delta )$, then the following equality holds:%
\begin{eqnarray*}
&&\frac{f(a,c)+f(a,d)+f(b,c)+f(b,d)}{4}-\frac{1}{\left( b-a\right) \left(
d-c\right) }\int_{a}^{b}\int_{c}^{d}f(x,y)dxdy \\
&&-\frac{1}{2}\left[ \frac{1}{b-a}\int_{a}^{b}\left[ f(x,c)+f(x,d)\right] dx+%
\frac{1}{d-c}\int_{c}^{d}\left[ f(a,y)+f(b,y)\right] dy\right] \\
&=&\frac{\left( b-a\right) \left( d-c\right) }{4}\int_{0}^{1}%
\int_{0}^{1}(1-2t)(1-2s)\frac{\partial ^{2}f}{\partial t\partial s}\left(
ta+(1-t)b,sc+(1-s)d\right) dtds.
\end{eqnarray*}
\end{lemma}

By using this equality and H\"{o}lder integral inequality for double
integrals, Sar\i kaya et al. obtained the following inequality:

\begin{theorem}
\label{T-3.1}Let $f:\Delta \subseteq 
%TCIMACRO{\U{211d} }%
%BeginExpansion
\mathbb{R}
%EndExpansion
^{2}\rightarrow 
%TCIMACRO{\U{211d} }%
%BeginExpansion
\mathbb{R}
%EndExpansion
$ be a partial differentiable mapping on $\Delta =\left[ a,b\right] \times %
\left[ c,d\right] $ in $%
%TCIMACRO{\U{211d} }%
%BeginExpansion
\mathbb{R}
%EndExpansion
^{2}$with $a<b$ and $c<d.$ If $\left\vert \frac{\partial ^{2}f}{\partial
t\partial s}\right\vert ^{q},q>1,$\ is convex function on the co-ordinates
on $\Delta $, then one has the inequalities:%
\begin{eqnarray}
&&\left\vert \frac{f(a,c)+f(a,d)+f(b,c)+f(b,d)}{4}-\frac{1}{\left(
b-a\right) \left( d-c\right) }\int_{a}^{b}\int_{c}^{d}f(x,y)dxdy-A\right\vert
\label{3-0} \\
&\leq &\frac{\left( b-a\right) \left( d-c\right) }{4(p+1)^{2/p}}\left[ \frac{%
\left\vert f_{st}(a,c)\right\vert ^{q}+\left\vert f_{st}(a,d)\right\vert
^{q}+\left\vert f_{st}(b,c)\right\vert ^{q}+\left\vert
f_{st}(b,d)\right\vert ^{q}}{4}\right] ^{1/q},  \notag
\end{eqnarray}%
where 
\begin{equation*}
A=\frac{1}{2}\left[ \frac{1}{b-a}\int_{a}^{b}\left[ f(x,c)+f(x,d)\right] dx+%
\frac{1}{d-c}\int_{c}^{d}\left[ f(a,y)+f(b,y)\right] dy\right] ,
\end{equation*}%
$1/p+1/q=1$ and $f_{st}=\frac{\partial ^{2}f}{\partial t\partial s}.$
\end{theorem}

If Theorem \ref{T-3.1} are resulted again by using the inequality (\ref%
{3C3-1}), then we get the following result:

\begin{theorem}
Let $f:\Delta \subseteq 
%TCIMACRO{\U{211d} }%
%BeginExpansion
\mathbb{R}
%EndExpansion
^{2}\rightarrow 
%TCIMACRO{\U{211d} }%
%BeginExpansion
\mathbb{R}
%EndExpansion
$ be a partial differentiable mapping on $\Delta =\left[ a,b\right] \times %
\left[ c,d\right] $ in $%
%TCIMACRO{\U{211d} }%
%BeginExpansion
\mathbb{R}
%EndExpansion
^{2}$with $a<b$ and $c<d.$ If $\left\vert \frac{\partial ^{2}f}{\partial
t\partial s}\right\vert ^{q},q>1,$\ is convex function on the co-ordinates
on $\Delta $, then one has the inequalities:%
\begin{eqnarray}
&&\left\vert \frac{f(a,c)+f(a,d)+f(b,c)+f(b,d)}{4}-\frac{1}{\left(
b-a\right) \left( d-c\right) }\int_{a}^{b}\int_{c}^{d}f(x,y)dxdy-A\right%
\vert   \label{3-1} \\
&\leq &\frac{\left( b-a\right) \left( d-c\right) }{4^{1+1/p}(p+1)^{2/p}}%
\left\{ \left[ \frac{4\left\vert f_{st}(a,c)\right\vert ^{q}+2\left\vert
f_{st}(a,d)\right\vert ^{q}+2\left\vert f_{st}(b,c)\right\vert
^{q}+\left\vert f_{st}(b,d)\right\vert ^{q}}{36}\right] ^{1/q}\right.  
\notag \\
&&+\left[ \frac{2\left\vert f_{st}(a,c)\right\vert ^{q}+\left\vert
f_{st}(a,d)\right\vert ^{q}+4\left\vert f_{st}(b,c)\right\vert
^{q}+2\left\vert f_{st}(b,d)\right\vert ^{q}}{36}\right] ^{1/q}  \notag \\
&&+\left[ \frac{2\left\vert f_{st}(a,c)\right\vert ^{q}+4\left\vert
f_{st}(a,d)\right\vert ^{q}+\left\vert f_{st}(b,c)\right\vert
^{q}+2\left\vert f_{st}(b,d)\right\vert ^{q}}{36}\right] ^{1/q}  \notag \\
&&\left. +\left[ \frac{\left\vert f_{st}(a,c)\right\vert ^{q}+2\left\vert
f_{st}(a,d)\right\vert ^{q}+2\left\vert f_{st}(b,c)\right\vert
^{q}+4\left\vert f_{st}(b,d)\right\vert ^{q}}{36}\right] ^{1/q}\right\} , 
\notag
\end{eqnarray}%
where 
\begin{equation*}
A=\frac{1}{2}\left[ \frac{1}{b-a}\int_{a}^{b}\left[ f(x,c)+f(x,d)\right] dx+%
\frac{1}{d-c}\int_{c}^{d}\left[ f(a,y)+f(b,y)\right] dy\right] ,
\end{equation*}%
$1/p+1/q=1$ and $f_{st}=\frac{\partial ^{2}f}{\partial t\partial s}.$
\end{theorem}

\begin{proof}
Using Lemma \ref{L-3.1} and the inequality (\ref{3C3-1}), we find%
\begin{eqnarray}
&&\left\vert \frac{f(a,c)+f(a,d)+f(b,c)+f(b,d)}{4}-\frac{1}{\left(
b-a\right) \left( d-c\right) }\int_{a}^{b}\int_{c}^{d}f(x,y)dxdy-A\right%
\vert   \label{*} \\
&\leq &\frac{\left( b-a\right) \left( d-c\right) }{4}\int_{0}^{1}%
\int_{0}^{1}\left\vert 1-2t\right\vert \left\vert 1-2s\right\vert \left\vert
f_{st}\left( ta+(1-t)b,sc+(1-s)\right) \right\vert dtds  \notag
\end{eqnarray}%
\begin{eqnarray}
&\leq &\frac{\left( b-a\right) \left( d-c\right) }{4}\left\{ \left(
\int_{0}^{1}\int_{0}^{1}ts\left\vert 1-2t\right\vert ^{p}\left\vert
1-2s\right\vert ^{p}dtds\right) ^{1/p}\right.   \notag \\
&&\times \left( \int_{0}^{1}\int_{0}^{1}ts\left\vert f_{st}\left(
ta+(1-t)b,sc+(1-s)\right) \right\vert ^{q}dtds\right) ^{1/q}  \notag \\
&&+\left( \int_{0}^{1}\int_{0}^{1}t(1-s)\left\vert 1-2t\right\vert
^{p}\left\vert 1-2s\right\vert ^{p}dtds\right) ^{1/p}  \notag \\
&&\times \left( \int_{0}^{1}\int_{0}^{1}t(1-s)\left\vert f_{st}\left(
ta+(1-t)b,sc+(1-s)\right) \right\vert ^{q}dtds\right) ^{1/q}  \notag \\
&&+\left( \int_{0}^{1}\int_{0}^{1}(1-t)s\left\vert 1-2t\right\vert
^{p}\left\vert 1-2s\right\vert ^{p}dtds\right) ^{1/p}  \notag \\
&&\times \left( \int_{0}^{1}\int_{0}^{1}(1-t)s\left\vert f_{st}\left(
ta+(1-t)b,sc+(1-s)\right) \right\vert ^{q}dtds\right) ^{1/q}  \notag \\
&&+\left( \int_{0}^{1}\int_{0}^{1}(1-t)(1-s)\left\vert 1-2t\right\vert
^{p}\left\vert 1-2s\right\vert ^{p}dtds\right) ^{1/p}  \notag \\
&&\left. \times \left( \int_{0}^{1}\int_{0}^{1}(1-t)(1-s)\left\vert
f_{st}\left( ta+(1-t)b,sc+(1-s)\right) \right\vert ^{q}dtds\right)
^{1/q}\right\} .  \notag
\end{eqnarray}%
Since $\left\vert f_{st}\right\vert ^{q}$ is convex function on the
co-ordinates on $\Delta $, we have for all $t,s\in \left[ 0,1\right] $%
\begin{eqnarray}
&&\left\vert f_{st}\left( ta+(1-t)b,sc+(1-s)\right) \right\vert ^{q}
\label{*-2} \\
&\leq &ts\left\vert f_{st}\left( a,c\right) \right\vert
^{q}+t(1-s)\left\vert f_{st}\left( a,d\right) \right\vert
^{q}+(1-t)s\left\vert f_{st}\left( a,c\right) \right\vert
^{q}+(1-t)(1-s)\left\vert f_{st}\left( a,c\right) \right\vert ^{q}  \notag
\end{eqnarray}%
for all $t,s\in \left[ 0,1\right] .$ Further since%
\begin{eqnarray}
\int_{0}^{1}\int_{0}^{1}ts\left\vert 1-2t\right\vert ^{p}\left\vert
1-2s\right\vert ^{p}dtds &=&\int_{0}^{1}\int_{0}^{1}t(1-s)\left\vert
1-2t\right\vert ^{p}\left\vert 1-2s\right\vert ^{p}dtds  \notag \\
&=&\int_{0}^{1}\int_{0}^{1}(1-t)s\left\vert 1-2t\right\vert ^{p}\left\vert
1-2s\right\vert ^{p}dtds  \label{3-2} \\
&=&\int_{0}^{1}\int_{0}^{1}(1-t)(1-s)\left\vert 1-2t\right\vert
^{p}\left\vert 1-2s\right\vert ^{p}dtds  \notag \\
&=&\frac{1}{4\left( p+1\right) ^{2}},
\end{eqnarray}%
a combination of (\ref{*}) - (\ref{3-2}) immediately gives the required
inequality (\ref{3-1}).
\end{proof}

\begin{remark}
Since $\eta :\left[ 0,\infty \right) \rightarrow 
%TCIMACRO{\U{211d} }%
%BeginExpansion
\mathbb{R}
%EndExpansion
,\eta (x)=x^{s},0<s\leq 1,$ is a concave function, for \ all $u,v\geq 0$ we
have%
\begin{equation*}
\eta \left( \frac{u+v}{2}\right) =\left( \frac{u+v}{2}\right) ^{s}\geq \frac{%
\eta (u)+\eta (v)}{2}=\frac{u^{s}+v^{s}}{2}.
\end{equation*}%
From here, we get%
\begin{eqnarray}
I &=&\left\{ \left[ \frac{4\left\vert f_{st}(a,c)\right\vert
^{q}+2\left\vert f_{st}(a,d)\right\vert ^{q}+2\left\vert
f_{st}(b,c)\right\vert ^{q}+\left\vert f_{st}(b,d)\right\vert ^{q}}{36}%
\right] ^{1/q}\right.   \notag \\
&&+\left[ \frac{2\left\vert f_{st}(a,c)\right\vert ^{q}+\left\vert
f_{st}(a,d)\right\vert ^{q}+4\left\vert f_{st}(b,c)\right\vert
^{q}+2\left\vert f_{st}(b,d)\right\vert ^{q}}{36}\right] ^{1/q}  \notag \\
&&+\left[ \frac{2\left\vert f_{st}(a,c)\right\vert ^{q}+4\left\vert
f_{st}(a,d)\right\vert ^{q}+\left\vert f_{st}(b,c)\right\vert
^{q}+2\left\vert f_{st}(b,d)\right\vert ^{q}}{36}\right] ^{1/q}  \notag \\
&&\left. +\left[ \frac{\left\vert f_{st}(a,c)\right\vert ^{q}+2\left\vert
f_{st}(a,d)\right\vert ^{q}+2\left\vert f_{st}(b,c)\right\vert
^{q}+4\left\vert f_{st}(b,d)\right\vert ^{q}}{36}\right] ^{1/q}\right\}  
\notag
\end{eqnarray}%
\begin{eqnarray}
&&\leq 2\left\{ \left[ \frac{6\left\vert f_{st}(a,c)\right\vert
^{q}+3\left\vert f_{st}(a,d)\right\vert ^{q}+6\left\vert
f_{st}(b,c)\right\vert ^{q}+3\left\vert f_{st}(b,d)\right\vert ^{q}}{72}%
\right] ^{1/q}\right.   \notag \\
&&\left. +\left[ \frac{3\left\vert f_{st}(a,c)\right\vert ^{q}+6\left\vert
f_{st}(a,d)\right\vert ^{q}+3\left\vert f_{st}(b,c)\right\vert
^{q}+6\left\vert f_{st}(b,d)\right\vert ^{q}}{72}\right] ^{1/q}\right\}  
\notag
\end{eqnarray}%
\begin{equation*}
\leq 4\left\{ \left[ \frac{\left\vert f_{st}(a,c)\right\vert ^{q}+\left\vert
f_{st}(a,d)\right\vert ^{q}+\left\vert f_{st}(b,c)\right\vert
^{q}+\left\vert f_{st}(b,d)\right\vert ^{q}}{16}\right] ^{1/q}\right. 
\end{equation*}%
Thus we obtain%
\begin{eqnarray*}
&&\frac{\left( b-a\right) \left( d-c\right) }{4^{1+1/p}(p+1)^{2/p}}I \\
&\leq &\frac{\left( b-a\right) \left( d-c\right) }{4^{1+1/p}(p+1)^{2/p}}%
4\left\{ \left[ \frac{\left\vert f_{st}(a,c)\right\vert ^{q}+\left\vert
f_{st}(a,d)\right\vert ^{q}+\left\vert f_{st}(b,c)\right\vert
^{q}+\left\vert f_{st}(b,d)\right\vert ^{q}}{16}\right] ^{1/q}\right\}  \\
&\leq &\frac{\left( b-a\right) \left( d-c\right) }{4(p+1)^{2/p}}\left\{ %
\left[ \frac{\left\vert f_{st}(a,c)\right\vert ^{q}+\left\vert
f_{st}(a,d)\right\vert ^{q}+\left\vert f_{st}(b,c)\right\vert
^{q}+\left\vert f_{st}(b,d)\right\vert ^{q}}{4}\right] ^{1/q}\right\} .
\end{eqnarray*}%
This show us that the inequality (\ref{3-1}) is better than the inequality (%
\ref{3-0}).
\end{remark}

\end{document}